\documentclass[11pt,a4paper,reqno]{amsart}

\usepackage{mathrsfs} 

\usepackage{amsmath,amssymb,graphics,epsfig,color,enumerate,psfrag}
\usepackage{dsfont}  
\usepackage{verbatim}
\newtheorem{Theorem}{Theorem}

\newtheorem{Lemma}{Lemma}
\newtheorem{Proposition}{Proposition}

\theoremstyle{definition}
\newtheorem{Remark}{Remark}

\newtheorem{Definition}{Definition}

\newtheorem{Assumption}{Assumption}
\newtheorem{Conjecture}{Conjecture}

\newcommand{\cE}{\ensuremath{\mathcal E}}
\newcommand{\cF}{\ensuremath{\mathcal F}}

\newcommand{\cO}{\ensuremath{\mathcal O}}

\newcommand{\cS}{\ensuremath{\mathcal S}}

\newcommand{\bbE}{{\ensuremath{\mathbb E}} }

\newcommand{\bbG}{{\ensuremath{\mathbb G}} }

\newcommand{\bbN}{{\ensuremath{\mathbb N}} }

\newcommand{\bbP}{{\ensuremath{\mathbb P}} }

\newcommand{\bbV}{{\ensuremath{\mathbb V}} }

\newcommand{\bbZ}{{\ensuremath{\mathbb Z}} }

\let\a=\alpha \let\b=\beta   \let\d=\delta  \let\e=\varepsilon
 \let\g=\gamma       
      \let\o=\omega      
   \let\t=\tau

\author[V. Silvestri]{Vittoria Silvestri}
\address{Vittoria Silvestri. 
Universit{\`a} di Roma La Sapienza, Roma, Italy.}
\email{silvestri@mat.uniroma1.it}

\title{Internal DLA on cylinder graphs: fluctuations and mixing}

\begin{document}
\maketitle

\begin{abstract}
We use coupling ideas introduced in \cite{levine2018long} to show that an IDLA process on a cylinder graph $G\times \bbZ$ forgets a typical initial profile in $\mathcal{O}( N\sqrt{\tau_N} (\log \! N)^2 )$ steps for large~$N$, where $N$ is the size of the base graph $G$, and $\tau_N$ is the total variation mixing time of a simple random walk on $G$. The main new ingredient is a maximal fluctuations bound for IDLA on $G\times \mathbb{Z}$ which only relies on the mixing properties of the base graph~$G$ and the Abelian property. 
\end{abstract}


\section{Introduction}
Internal Diffusion Limited Aggregation (IDLA) is a mathematical model for corrosion phenomena, introduced by Meakin and Deutch \cite{meakin1986formation} and independently by Diaconis and Fulton \cite{diaconis1991growth}. It models the growth of a cluster of particles as being governed by the harmonic measure on its boundary seen from an internal point. 
More precisely, let $\bbG = (\bbV , \bbE )$ denote an infinite, locally finite graph with a marked vertex $Z_0 \in \bbV$, and assume that a simple random walk on $\bbG$ starting from $Z_0$ exits every finite set containing $Z_0$ in finite time almost surely. We define an increasing family $(A(t))_{t\geq 0}$ of subsets of $\bbV$ as follows. Set $A(0) = \{ Z_0\}$, and recursively define 
	\[ A(t) = A(t-1) \cup \{ Z_t \} , \]
where $Z_t$ denotes the exit location from $A(t-1)$ of a simple random walk on $\bbG$ starting from $Z_0$, independent of everything else. Then $(A(t))_{t\geq 0}$ is a Markov chain on the space of connected subsets of $\bbG$, whose long-term behaviour is of interest.  
The case $\bbG = \bbZ^d$, $Z_0=0$ is by now well understood: as $t\to\infty$, $A(t)$ has been shown to converge to a Euclidean ball \cite{lawler1992internal} with logarithmic fluctuations away from it in dimension $d\geq 2$  \cite{asselah2013logarithmic,asselah2013sublogarithmic,jerison2012logarithmic,jerison2013internal,asselah2014lower,jerison2014internal,jerison2014internal2}. 
It was recently showed that the limiting shape is stable under the following perturbation: rather than starting random walks at $0$, sample the starting point of the $k^{th}$ walk uniformly on $A(k-1)$ \cite{benjamini2017internal}. Bounding the fluctuations of IDLA with uniform starting point remains an open problem. 

Let $G=(V,E)$ be a finite, connected, non--oriented graph on $N$ vertices. We take $\bbG = G \times \bbZ$, that is $\bbG = (\bbV , \bbE)$ with 
	\[ \begin{split} 
	& \bbV = \{ (v,y) : v\in G , \, y\in \bbZ \} 
	\\ & \bbE = \{ ((v,y),(v',y') ) : \;  v\sim v' \mbox{ and }y=y',  \mbox{ or } v=v' \mbox{ and } |y-y'|=1 \} ,
	\end{split} 
	\]
where $\sim$ denotes the adjacency relation in $G$. 
We refer to the first and second coordinate as the horizontal and vertical coordinate respectively. For $y \in \bbZ$ we call the set $G \times \{ y\}$ the $y^{th}$ level, while $R_y = G \times(-\infty , y]$ will denote the rectangle of height $y$. 
\begin{Definition}[Simple random walk on $\bbG$]
A simple random walk on $\bbG$ is, for us, a discrete--time Markov chain which at each step moves either in the horizontal or in the vertical coordinate with probability $1/2$. A vertical move consists in stepping from $(v,y)$ to $(v,y\pm 1)$ with equal probability, while a horizontal move is a lazy simple random walk step on $G$, i.e. $(v,y) \to (v,y)$ with probability $1/2$ and $(v,y) \to (w,y)$ with probability $1/(2\deg v)$ if $w$ is a neighbour of $v$ in $G$. 
\end{Definition}

Let $\pi_N$ denote the stationary distribution of a simple random walk on $G$, that is 
	\[ \pi_N (v) = \frac{\deg v }{2|E|} , \qquad v \in V \]
where $\deg v$ is the degree of $v$ in $G$. We make the following assumption throughout. 
\begin{Assumption}\label{assumption}
There exist finite constants $\d, \d ' >0 $, independent of $N$, such that 
	\begin{equation} \label{mu_spread}
	  \frac{\d}{N} \leq  \pi_N (v) \leq \frac{\d '}{N}  , \qquad \forall v \in V ,  
	\end{equation}
for all $N\geq 1$. 
\end{Assumption}
Graphs satisfying the above assumption are called \emph{quasi-regular}. 
Example include the complete graph, expander graphs, the $d$-dimensional torus,  and any graph with vertices of comparable degree. In particular, graphs in this class can have a wide range of mixing behaviour for large $N$ (see the comments after the statement of Theorem \ref{th:main_teo}). 

We remark that Assumption \ref{assumption} is not needed for our arguments to work: we choose to make it as it identifies the class of base graphs for which our approach gives best results. \smallskip

\textbf{Notation.} 
All constants in this note are allowed to depend on $\d , \d '$, and we suppress this dependence from the notation.

\begin{Definition}[IDLA on $\bbG$]
Let $A(0)$ be any connected subset of $\bbV $ containing the rectangle $R_0 = \{ (v,y) : y\leq 0\}$ and finitely many sites above level $0$. 
At each discrete time $t\geq 1$, a simple random walk is released from level $0$ at a random location sampled from $\pi_N$, and its exit location $Z_t$ from $A(t-1)$ is added to the cluster by setting 
	\[ A(t) = A(t-1) \cup \{ Z_t\}. \] 
\end{Definition}
Note that this is equivalent to adding a site to the cluster according to the harmonic measure on the cluster's external boundary seen from level $-\infty $. When $A(0) = R_0$ we say that the process starts from flat. 

It is common to describe the IDLA dynamics as an interacting particle system. Particles are released one by one from level $0$ according to $\pi_N$, and perform simple random walks until they hit an unexplored location, where they settle. The cluster is then the set of sites containing settled particles. 

Let 
	\[ \Omega = \{ A \subset \bbV : A = R_0 \cup F \mbox{ for some finite }F \}  \]
denote the state space of IDLA on $\mathbb{G}$, 
and note that the IDLA dynamics defines a transient Markov chain on $\Omega$. The following shift procedure makes it recurrent. 
\begin{Definition}[Shifted IDLA]\label{def:shifted}
Let $\cS : \Omega \to \Omega $ be  defined as follows. For a given cluster $A \in \Omega$, let
	\begin{equation}\label{kA}
	 k_A := \max \{ k\geq 0 : 
	R_k \subseteq A \}  
	\end{equation}
denote the height of the maximal filled rectangle contained in $A$.
The \emph{downshift} of $A$ is the cluster 
	\[ \cS (A) := \{ (v,y-k_A) : (v,y) \in A \}. \]
Note that $k_{\cS (A)} = 0$, so $\cS (\cS (A)) = \cS(A)$.   
If  $(A(t))_{t\geq 0}$ is an IDLA process in $\Omega$, we set
	\[ A^*(t) := \cS ( A(t)) \]
for all $t\geq 0$, and refer to $(A^*(t))_{t\geq 0}$ as the \emph{shifted IDLA} process associated to $(A(t))_{t\geq 0}$. 
\end{Definition}
Shifted IDLA is itself a Markov chain\footnote{To see this, it is convenient to modify the definition of the IDLA process $(A(t))_{t\geq 0}$ by releasing the $t^{th}$ particle from level $k_{A(t-1)}$ rather than from level $0$, with $k_{A(t-1)}$ defined as in \eqref{kA}. This clearly does not change the distribution of the process.} on the countable state space $\Omega$. It is easy to see that this chain is irreducible, since all states communicate with the flat configuration $R_0$. As in \cite{levine2018long}, it can be shown that shifted IDLA is positive recurrent, so it has a stationary distribution $\mu_N$ (cf.\ Remark \ref{rem:positiverec} below). If a random cluster $A$ is distributed according to $\mu_N$, we say that $A$ is a \emph{stationary cluster}. 

\smallskip
The aim of this note is to address the following questions: 
\begin{itemize}
\item What do stationary clusters look like? 
\item How long does shifted IDLA take to forget a stationary initial state? 
\end{itemize}
The first question concerns fluctuation bounds for IDLA clusters, while the second one looks as the mixing properties of shifted IDLA as a Markov chain. Both questions were answered in \cite{levine2018long} for the cycle base $G=\bbZ_N$, which is by now well understood. Here we generalise the analysis to quasi--regular base graphs, for which far less is known. In particular, while \cite{levine2018long} makes use of optimal fluctuation bounds which are specific to the case $G=\bbZ_N$, here we observe that near-optimal fluctuation bounds are not needed to obtain a near-optimal mixing result (see comments after the statement of Theorem \ref{th:main_teo}). 

\smallskip 

In order to state our results, recall that a \emph{coupling} of two simple random walks on $G$ is a process $(\o (k) , \o'(k))_{k\geq 0}$ such that both marginals are simple random walks on $G$, possibly with different initial distributions.  
	
Recall that for two probability measures $\mu , \nu$ on a countable state space $\Omega$, their total variation (TV) distance is defined as 
	\[ \| \mu - \nu \|_{TV} = \sup_{A\subseteq \Omega } | \mu (A) - \nu (A) | . \]
\begin{Definition}[TV mixing time]
Let $\o = (\o (k))_{k\geq 0}$, $\o ' = (\o '(k))_{k\geq 0} $ denote two simple random walks on $G$ starting from $\o (0) =v$, $\o '(0) = v'$. Then, if $P^k_v$ and $P^k_{v'}$ denote the distribution of $\o (k) $ and $\o '(k)$ respectively, we define the $\e$-TV mixing time by 
	\[ \tau_N (\e ) = \inf\{ k\geq 0 : \max_{v ,v' \in V} \| P^k_{v} - P^k_{v'} \|_{TV} \leq \e \} ,\]
and write $\tau_N$ in place of $\tau_N(1/2)$ for brevity.
\end{Definition}
	

For $A \in \Omega$, let 
	\[h(A) = \max \{y: (v,y) \in A \mbox{ for some }v\in G \}\] 
denote the height of $A$, so that $h(A) \geq 0$ for all $A \in \Omega$. 
Our first result bounds the height of stationary clusters in terms of $\tau_N$, and it implies that the stationary distribution~$\mu_N$ concentrates on a small subset of the state space $\Omega$. 

\begin{Theorem}\label{th:typical_height}
Assume \eqref{mu_spread}. Then for any $\g >0$, there exists a constant $ c_{\g } $, depending only on $\g , \d , \d'$, such that 
	\begin{equation}\label{typheight}
	\mu_N \Big( \big\{ A \in \Omega : h(A) > c_{\g }  \sqrt{\tau_N} (\log N)^2 \big\} \Big) 
	\leq  N^{-\g}  .
	\end{equation}
for $N$ large enough. 
\end{Theorem}
A similar result holds for any \emph{lukewarm start} $\nu_N$ for shifted IDLA (cf.\ \cite{levine2018long}, Section~1), we omit the details. Note that the faster the mixing of the base graph, the smaller the height bound provided by Theorem \ref{th:typical_height}.
We remark that the proof of Theorem \ref{th:typical_height} only relies on the mixing properties of the base graph $G$ and the Abelian property of IDLA (cf.\ Section \ref{sec:abelian}), but does not require any Green's function estimate for the underlying random walks.  We have not tried to optimise the exponent of the logarithmic factor. 
\smallskip

We use the above control on stationary clusters to bound the time it takes for shifted IDLA to forget any such initial state. 

\begin{Theorem}\label{th:main_teo}
Assume \eqref{mu_spread}. Then for any $\g >0$ there exist a constant $d_{\g }$ and a set $\Omega_{\g } \subseteq \Omega$, depending only on $\g , \d , \d'$, such that 
	\[ \mu_N (\Omega_{\g } ) \geq 1-N^{-\g}\]
for $N$ large enough, and such that the following holds. For any $A_0, A'_0 \in \Omega_{\g }$ with $|A_0| = |A'_0|$, writing $t_{\g } = d_{\g } N \sqrt{\tau_N} (\log N)^2$ for brevity, we have	
\begin{equation}\label{eq:forget}
  \| P(t_{\g } ) - P'(t_{\g } ) \|_{TV} \leq N^{-\g} 
 \end{equation}
where $(A(t))_{t\geq 0}$ and $(A'(t))_{t\geq 0}$ are IDLA processes starting from $A_0$ and $A'_0$ respectively, and $P(t)$ and $P'(t)$ denote the laws of $A(t)$ and $A'(t)$.
\end{Theorem}
Thus IDLA forgets any stationary initial state in $\cO (N \sqrt{\tau_N} (\log N)^2)$ steps, with high probability. 
The reader should compare this result with Theorem 1.3 of \cite{levine2018long} for $G=\bbZ_N$, stating that \eqref{eq:forget} holds with $t_\g = d_\g N^2 \log N $. Although 
the stationary height bound \eqref{typheight} in Theorem \ref{th:typical_height} is far from optimal in this case, we remark that the conclusion of Theorem~\ref{th:main_teo} only differs from the near-optimal result in \cite{levine2018long} by a logarithmic factor. 
We conjecture that this is the case for all quasi--regular graphs. 
\begin{Conjecture}
Under Assumption \ref{assumption}, $t_\gamma$ in Theorem \ref{th:main_teo} is optimal up to logarithmic factors.
\end{Conjecture}
In fact, as will become clear from the proof of Theorem \ref{th:main_teo}, our methods give a bound on the time it takes for an IDLA process to forget \emph{any} initial state. More explicitly, an IDLA process starting from an arbitrary profile $A_0 \in \Omega$ with height $h_0$ forgets its initial profile in $\cO ( N h_0 + N \sqrt{\t_N} (\log N)^2)$ steps, with high probability. 

\subsection*{Structure of the paper}
The remainder of the paper is organised as follows. 
In Section \ref{sec:abelian} we introduce the Abelian property of IDLA, and explain how it relates to couplings of IDLA processes. In Section \ref{sec:preliminary} we collect two preliminary results on the mixing properties of a simple random walk on $\mathbb{G}$.
In Section \ref{sec:big} we recall the coupling ideas introduced in \cite{levine2018long}, and adapt them to bound the maximal fluctuations for IDLA on $\bbG$ for polynomial times (cf.\  Theorem~\ref{th:Tpoly}). This is the core of the article. 
Finally, in Section \ref{sec:brief} we briefly explain how to deduce Theorems \ref{th:typical_height} and \ref{th:main_teo} along the lines of \cite{levine2018long}, leaving the details to the reader. 

\subsection*{Acknowledgement}
I am thankful to Lionel Levine for several insightful discussions, and for a careful reading of an earlier version of this note. I am also indebted with Alexandre Stauffer for pointing me to Pitman's construction of Markov chains optimal couplings, and for many stimulating discussions on this and related models. 

\section{On couplings of Abelian processes}\label{sec:abelian}
Internal DLA was introduced in the mathematical literature by Diaconis and Fulton \cite{diaconis1991growth} as an example of an Abelian growth model. In this section we introduce the Abelian property, and comment on the role it plays in the construction of a valid coupling of two IDLA processes. 
\subsection{The Abelian property}
There are several equivalent ways to state the Abelian property of IDLA: in this note we opt for the \emph{stack of instructions} point of view \cite{diaconis1991growth,rolla2012absorbing,levine2020preparation}. Let $\bbG$ be a given graph (in our case $\bbG = G \times \bbZ$), with vertex set $\bbV$ and edge set $\bbE$. To each vertex $x\in \bbV$ of the graph we associate an infinite stack of instructions 
	\[ \xi_x = ( \xi_x(k))_{k\geq 1},\] 
with $\xi_x (k)$ neighbour of $x$ in $\bbG$ for all $k\geq 1$. 
For the moment we think of the stack of instructions $\xi = (\xi_x)_{x\in \bbV}$ as fixed, so there is no randomness in the model.  

Given an initial configuration $\eta : V \to \bbN$ of particles, where $\eta (x) $ denotes the number of particles at $x$, we say that $\eta $ is unstable at $x$ if $\eta (x) \geq 2$. A configuration $\eta$ is unstable if there exists $x \in V$ such that $\eta $ is unstable at $x$. We stabilise $\eta $ as follows. At each discrete time step $t\geq 1$, we pick an unstable vertex $x \in V$, and move one particle from $x$ to an adjacent vertex according to the first unused instruction at $x$. More explicitly, if $k_t$ denotes the number of visits to $x$ up to and including time $t$, we move one particle from $x$ to $\xi_x(k_t)$. We refer to this operation as \emph{toppling} the vertex $x$. We keep toppling unstable vertices until $\eta$ becomes stable, that is $\eta (x) \leq 1$ for all $x\in \bbV$.  

A given sequence of topplings $\a = (x_ 1, x_2 \ldots x_n )$ is said to be \emph{legal} for a configuration $\eta$ if for any $2\leq k \leq n$ the vertex $x_k$ is unstable for the configuration obtained by toppling $\eta $ at $x_1 , x_2 \ldots x_{k-1}$. If a legal sequence $\a$ stabilises $\eta$, we write $\cS_\a (\eta )$ for the stabilization of $\eta$ according to $\a$. 
\begin{Theorem}[Abelian property]
If $\a = (x_1 , x_2 \ldots x_n)$ and $\beta = (y_1 , y_2 \ldots y_n )$ are two finite, legal stabilising sequences for $\eta$, then they are permutations one of the other, and $\cS_\a (\eta ) = \cS _\b (\eta )$. 
\end{Theorem}
In particular, this tells us that the stabilization of a given configuration $\eta$ according to a given stack of instructions $\xi = (\xi_x)_{x\in \bbV}$ does not depend on the sequence of topplings. This is particularly useful when randomness is introduced in the story. 

\subsection{Couplings preserving the Abelian property}
With IDLA in mind, we randomize the stack of instructions $\xi = ( \xi_x )_{x \in \bbV}$ as follows. Take $( \xi_x )_{x \in \bbV}$ to be i.i.d., each consisting of i.i.d.\ instructions. For example, for $\bbG = G\times \bbZ$ we could generate a random stack of instructions by letting a simple random walk on $\bbG$ (as described in the introduction) run forever, and recording its jumps. 

In this setting, the Abelian property tells us that if $\xi = ( \xi_x )_{x \in \bbV}$ and $\xi '= ( \xi '_x )_{x \in \bbV}$ are two identically distributed random stacks, then stabilizing a given configuration $\eta $ with respect to $\xi$ and $\xi '$ will result in identically distributed stable configurations. 

This suggests a way to couple two IDLA processes on $\bbG$. Indeed, we could first drop $n$ particles at level $0$ in an independent fashion according to $\pi_N$, and then see an IDLA process on $\bbG$ as the stabilization of such a particle configuration. Then to couple two IDLA processes starting from the same configuration it suffices to couple the corresponding stacks of instructions, which clearly preserves the Abelian property. 

While this is an explicit and natural way to couple two IDLA processes, we often want to describe the IDLA dynamics in terms of particle trajectories, as we did in the introduction. Thus we release particles one by one, and let them walk on $\bbG$ until they find an unoccupied site, where they settle. Note that recording the steps of the walks generate a stack of instruction on the visited sites. To couple two IDLA processes, then, rather than coupling the stacks of instructions one could instead couple the random walks trajectories. While this is a valid approach, and we will adopt it in this note, it should be implemented with some care. To see why, fix two arbitrary IDLA initial clusters $A(0)$ and $A'(0)$ (we are not ruling out that $A(0) = A'(0)$). To build $A(1)$, $A'(1)$ we start two simple random walks $\o = (\o (k))_{k\geq 0}$, $\o ' = (\o '(k))_{k\geq 0} $ from level $0$ in $A(0)$ and $A'(0)$ respectively. Imagine that the two walks are coupled in such a way that $\o '$ looks into the future of $\o$ in order to decide where to move. For example, if $G = \bbZ_N$ is the $N$-cycle, then we could set $\o '(k) - \o '(k-1) = \o (k+1) - \o (k) $ (mod $n$) for all $k\geq 1$. The walks are stopped upon exiting their respective clusters. Then, because  $\o '$ is allowed to look into the future of $\o$, $\o '$ might reveal stack instructions that $\o $ did not use. This creates correlations among subsequent walkers in the same IDLA process, since the sets of instructions they explore are not disjoint, and it thus invalidates the coupling. 

A simple way to overcome this problem is to force the random walks coupling to be co-adapted, that is a walker can look in the past of the other, but not in the future (see \cite{aldous1987strong,burdzy2000efficient,connor2008optimal} and references therein). This guarantees that subsequent walkers in the same IDLA process use disjoint sets of instructions. On the other hand, requiring a coupling to be co-adapted could, a priori, substantially increase the coupling time \cite{anil2001coupling}. To avoid this, in Section \ref{sec:preliminary} we build an explicit coupling that is not co-adapted, but with the property that different walkers in the same IDLA process use disjoint sets of instructions. 
Our construction follows Pitman \cite{pitman1976coupling} (see also \cite{griffeath1975maximal} and references therein), and it achieves a coupling time that is within a constant factor from the optimal one.

\section{Preliminaries}\label{sec:preliminary}
We collect here two preliminary results which were proved in \cite{levine2018long} for
 $G=\bbZ_N$. To start with, note that for any $\e >0$ it is easy to check that 
$\tau_N(\e ) \leq \lceil \log_2 (1/\e ) \rceil \tau_N $, 
which implies that
 	\begin{equation}\label{tt}
 	\tau_N (N^{-\g })  \leq 3\g \tau_N \log N 
 	\end{equation}
 for any $\g >0$. 
Let $\o = (\o (t))_{t\geq 0}$  be a simple random walk on $\mathbb{G}$, as described in the introduction, and let 
	\[ 
	 \tau^\bbZ_n := \inf \{ t\geq 0 : \o (t) \in G\times \{ n\} \} 
	 \]
denote the first time the walk reaches level $n$.  For $(v,y) $ vertex of $\bbG$ we denote by $\bbP_{(v,y)}$  the distribution of a simple random walk  on $\bbG$ starting from $(v,y)$. 
\begin{Proposition} \label{pr:mixG}
For any $\g >0$, it holds 
	\[ \max_{v\in V} \bbP_{(v,0)} (\tau_n^\bbZ < 3\g \tau_N \log N ) \leq N^{-\g } \]
for all $n \geq 10 \g  \sqrt{\tau_N} \log N . $
\end{Proposition}
This can be proved exactly as in \cite{levine2018long}, Lemma 3.1, replacing $N^2$ by $\tau_N$. It tells us that, by the time a simple random walk on $\mathbb{G}$ has travelled $\sqrt{\tau_N} \log N $ in the vertical coordinate, it has mixed well  in the horizontal one. 
\def\j{
\textcolor{green}{
\begin{proof}
By \eqref{t_epsilon}, it suffices to prove that 
	\[ \max_{x\in V} \bbP_{(x,0)} (\tau_n^y < \tau^x_{mix} \log N^{3\g} ) \leq N^{-\g } \]
for all $n \geq 10 \g \sqrt{\tau^x_{mix}} \log N . $
Let $(\hat{x} (t))_{t\geq 0} $ be the simple random walk on $G$ obtained by only looking at $x(t)$ when it moves. 
Similarly, let $(\hat{y} (t))_{t\geq 0} $ be the simple random walk on $\bbZ$ obtained by only looking at $y(t)$ when it moves.  Define
	\[ \hat{\t}^y_n := \inf\{ t\geq 0 : \hat{y} (t) =n \} .\]
For $i\geq 0$, let $G_i -1$ denote the number of moves in the $x$ coordinates before the $(i+1)^{th}$ move in the $y$ coordinate. Then $(G_i)_{i\geq 0}$ is a collection of i.i.d.\ Geometric random variables of parameter $1/2$, and we have 
	\[ x(\tau^y_n ) = \hat{x} (\hat{\tau}^y_n ) , \qquad 
	\tau^y_n  = \sum_{i=1}^{\hat{\tau}^y_n -1} G_i . \]
It follows that
	\begin{equation} \label{eq:tt*}
	 \begin{split} 
	\max_{x\in V }   \bbP_{(x,0)}  (  \tau^y_n & < \tau^x_{mix} \log N^{3\g}  ) 
	 = \hat{\bbP}_{0} \bigg( \sum_{i=1}^{\hat{\tau}^y_n -1} G_i   < \tau^x_{mix} \log N^{3\g} \bigg) \\
	& = \hat{\bbP}_{0} \bigg(  \sum_{i=1}^{\hat{\tau}^y_n -1} G_i <  \tau^x_{mix} \log N^{3\g } , \; \hat{\t}^y_n > 9 \g \tau^x_{mix} \log N \bigg) 
	\\ & \quad + \hat{\bbP}_{0} \bigg(  \sum_{i=1}^{\hat{\tau}^y_n -1} G_i < \tau^x_{mix} \log N^{3\g}  , \; \hat{\t}^y_n \leq  9 \g \tau^x_{mix} \log N  \bigg) \\
	& \leq \hat{\bbP}_{0} \bigg(  \sum_{i=1}^{8 \g \tau^x_{mix} \log N } G_i < 3 \g \tau^x_{mix} \log N  \bigg) +  \hat{\bbP}_0 \big( \hat{\t}^y_n \leq 9 \g  \tau^x_{mix} \log N  \big) . 
	\end{split} 
	\end{equation}
We estimate the two terms separately: the first one is controlled by large deviations estimates, while the second one by using the explicit expression for the moment generating function of $\hat{\t}^y_n $. More precisely, we have 
	\[ \begin{split} 
	\hat{\bbP}_{0} \bigg(  \sum_{i=1}^{8 \g \tau^x_{mix} \log N } G_i < 3\g \tau^x_{mix} \log N \bigg) & = 
	\hat{\bbP}_{0} \bigg( 2 ^{ - \sum_{i=1}^{8 \g \tau^x_{mix} \log N} G_i } >  2^{-  3 \g \tau^x_{mix} \log N }  \bigg)
	\\ & 
	\leq  \Big[ 2 \bbE ( 2^{- G_1} )^{\frac{8}{3}} \Big]^{3\g \tau^x_{mix} \log N} 
	\leq  \Big( \frac{2}{9} \Big)^{3 \g \tau^x_{mix} \log N } 
	\leq N^{-3\g} .
	\end{split}
	\]
For the second term, it is simple to check that, for $z \in (0,1)$,
	\[ \hat{\bbE}_0 \big( z^{\hat{\t}^y_n } \big) = \frac{1-\sqrt{1-z^2}}{z} . \]
This gives, for any $z \in (0,1)$,
	\[ \begin{split} 
	\hat{\bbP}_0 \big( \hat{\t}^y_n < 9 \g \tau^x_{mix} \log N   \big) 
	& \leq \hat{\bbE}_0 \big( z^{\hat{\t}^y_n } \big)  z^{-9 \g  \tau^x_{mix} \log N  } \\
	& = \bigg( \frac{1-\sqrt{1-z^2}}{z} \bigg)^n \cdot \bigg( \frac{1}{z} \bigg)^{9 \g  \tau^x_{mix} \log N  } . 
	\end{split} \]
Choose $z\in (0,1)$ of the form $z^2 = 1-1/\a^2$ for some $\a \to +\infty $ as $N\to\infty$, to have that 
	\[ \begin{split} 
	\bigg( \frac{1-\sqrt{1-z^2}}{z} \bigg)^n \cdot \bigg( \frac{1}{z} \bigg)^{9 \g  \tau^x_{mix} \log N  } 
	& \leq \Big( 1-\frac{1}{\a} \Big)^{n/2} \, \Big( 1-\frac{1}{\a^2} \Big)^{-\frac{9}{2} \g \tau^x_{mix} \log N } \\ & 
	\leq \exp \Big( -\frac{n}{2\a} + \frac{5\g }{\a^2} \tau^x_{mix} \log N \Big) .
	\end{split} \]
To have the far r.h.s.\ smaller than, say,  $N^{-5\g/4 }$ it suffices to take 
	\[ n \geq \frac{10 \g }{\a} \tau^x_{mix} \log N +  \frac{5}{2}\g \a \log N .\]
Optimizing over $\a$ suggests to take $\a = 2\sqrt{\tau^x_{mix}}$, to get $n \geq 10 \g  \sqrt{\tau^x_{mix}} \log N$. 
\end{proof}
}
}

Proposition \ref{pr:mixG} can be used to show that, as long as the walkers have time to mix in the horizontal coordinate before exiting the cluster, the IDLA process does not depend too much on their initial positions. 
\begin{Proposition} \label{pr:coupling}
Fix any $T>0 $  such that $T\leq N^m$ for some finite $m \in \bbN $. Let $\{ (v_i,y_i) \}_{1\leq i \leq T }$ and $\{ (v'_i , y'_i ) \}_{1\leq i \leq T }$ denote two fixed collections of vertices of $\mathbb{G}$ such that $y_i , y'_i \leq 0$ for all $i$. Let, moreover, $A=(A(t))_{t\leq T}$ and $A'=(A'(t))_{t\leq T}$ be two IDLA processes with starting configurations $A(0) = A'(0)$, such that the $i^{th}$ walker in $A$ (respectively $A'$) starts from $(v_i,y_i)$ (respectively from $(v_i',y_i')$). Then there exists a coupling of $A, A'$ such that the following holds. For any $\g >0$ there exists a finite constant $b_{\g , m }$, depending only on $\g$ and $m$, such that if 
	\begin{equation}\label{A0}
	 A(0) = A'(0) \supseteq R_{b_{\g , m } \sqrt{\t_N} \log N } , 
	 \end{equation}
then 
	\begin{equation} \label{AT}
	 \bbP ( A(t) = A'(t) \mbox{ for all }t\leq T ) \geq 1 - N^{-\g }. 
	 \end{equation}
\end{Proposition}
This tells us that, as long as an IDLA cluster contains a tall enough rectangle,  releasing  $T$ walkers from fixed initial locations below level $0$ or from $\pi_N$-distributed locations at level $0$ results in the same final cluster with high probability. We prove this via a coupling argument, showing that, as long as the rectangle is tall enough, the walkers have time to mix before leaving the cluster. 

\begin{proof}
Let $\g , m $ be given as in the statement, and set $\g ' = \g + m +1$. 
Recall that $R_n$ denotes the rectangle of height $n$, and assume that  $A(0) = A'(0) \supseteq R_n$ for $n =10 \g ' \sqrt{\t_N} \log N$. 

For $i\leq T$, let $\o_i = ( (v_i(k) , y_i(k) ) )_{k\geq 0}$ and $\o'_i = ( ( v '_i(k) , y'_i(k) ) )_{k\geq 0}$ denote the simple random walk trajectories of the $i^{th} $ walkers starting from $(v_i , y_i )$ and $(v_i' , y_i')$ respectively. These are coupled as follows. 
If, say, $y_i < y_i '$, then $\o '_i$ stays in place until $\o_i$ reaches level $y_i'$ (and vice versa if $y_i > y_i'$). We can therefore assume that $y_i = y_i'$ without loss of generality. We make the walks move together in the vertical coordinate, so that 
 $y_i(0) = y_i '(0) $ implies that $y_i(k) = y_i '(k)$ for all $k\geq 0$. 
 
We explain how to generate the trajectories of $\o_1 , \o_1'$ in such a way that they coincide upon reaching level $n$, with high probability. Our construction is based on 
\cite{pitman1976coupling}. The same construction can be used for the remaining walks. 

Under the assumption that both $\o_1 $ and $\o_1'$ start from level $y_1$, generate on the same probability space an auxiliary process $Y=(Y_t)_{t\geq 0}$ on $\bbZ$ starting from $Y_0 = y_1$. The process sojourns at each site for a Geometric random time of mean $2$, and then jumps to a uniform neighbour. This will be the projection of our walkers' trajectories onto the vertical coordinate. The process stops at 
	\[ \t_n^Y = \inf \{ t\geq 0 : Y_t =n\} ,\]
that is upon reaching level $n$. 
Let 
	\[ \t_n^X = \sum_{t=1}^{\tau_n^Y} \mathbf{1} (Y_t = Y_{t-1} ) \]
denote the amount of time the walk stayed still. Then, if $s \geq 3 \g' \tau_N \log N$ is an integer, we know from \eqref{tt} that 
	\[ \| P^s_{v_1} - P^s_{v_1'} \|_{TV} \leq N^{-\g '} \]
for $N$ large enough. 
Hence, conditional on $\{ \t_n^X =s\}$, we can generate on the same probability space a pair of random variables $(X,X') \in G\times G$ such that $X \sim P^s_{v_1}$, $X' \sim P^s_{v_1'}$ and 
$ \bbP (X \neq X') \leq N^{-\g '}$. 
Set 
	\[ \o (\tau_n^Y ) = (X,n) , \qquad \o ' (\tau_n^Y ) = (X' ,n).\]
Finally, given everything else, generate the horizontal steps of the two walkers as simple random walk bridges from $v_1 $ to $X$ (respectively from $v_1'$ to $X'$) in time $s$, independent of each other. 
The trajectories of $\o , \o'$ up to time $\tau_n^Y$ can then be built by following the $Y$ process when it jumps, and following the random walk bridges at times when $Y$ stays in place. 

The fact that this is a valid coupling up to time $\tau_n^Y$ follows from \cite{pitman1976coupling}. Moreover, since $\tau_n^Y$ is a randomised stopping time for both $\o_1$ and $\o_1'$ by construction, the strong Markov property is in force \cite{pitman1976coupling,connor2008optimal}, and hence we can make the walkers stick together after time $\t_n^Y$ by setting 
	\[ \o (k) = \o '(k) \;\; \mbox{ for } k\geq \tau_n^Y . \]
It follows that, provided the coupling is successful, $\o_1$ and $\o_1'$ will exit $A(0)$ and $A'(0)$ at the same location, thus implying that $A(1) = A'(1)$. By Proposition \ref{pr:mixG} the probability of the coupling failing is bounded by 
	\[ \bbP ( \o_1 \mbox{ and } \o '_1  \mbox{ do not couple within } R_n ) 
	 \leq \bbP (  \t_n^Y < 3\g ' \tau_N \log N ) +  
	 \bbP ( X \neq X' ) 
	 \leq 2 N^{-\g '} \]
for $N$ large enough. This procedure can be iterated for $(\o_i , \o_i ')_{2\leq i \leq T}$, and if all the couplings are successful then $A(t) = A'(t) $ for all $t\leq T$. The probability of failure is 
	\[ \bbP ( \exists t \leq T \mbox{ such that } A(t) \neq A'(t) ) \leq 
	2T N^{-\g '} \leq N^{\g} \]
for $N$ large enough. Thus \eqref{AT} holds with $b_{\g ,m} = 10(\g + m +1)$. 
\end{proof}
\smallskip 

\section{IDLA maximal fluctuations}\label{sec:big}
In this section we adapt the \emph{water level coupling} introduced in \cite{levine2018long} to prove the following result on the maximal fluctuations of IDLA on $\mathbb{G}$. 
\begin{Theorem}[Maximal fluctuations for polynomial times] \label{th:Tpoly}
Assume \eqref{mu_spread}. Let $(A(t))_{t\geq 0}$ be an IDLA process on $\mathbb{G}$ starting from the flat configuration $A(0)= R_0 $. 
Suppose that $T=T(N) \leq  N^m$ for some finite $m\in \bbN$. Then for any $\g >0$ there exists a constant  $C_{\g , m}$, depending only on  $\g , \d , \d ' , m $, such that 
	\begin{equation}\label{eq:max} 
	 \bbP \Big(  R_{\frac{T}{N} - C_{\g ,m} \sqrt{\t_N} ( \log N )^2 }
	\subseteq  A(T) \subseteq
	R_{\frac{T}{N} + C_{\g,m} \sqrt{\t_N} ( \log N )^2 } \Big) \geq 1-N^{-\g}  
	\end{equation}
for $N$ large enough. 
\end{Theorem}
The above theorem, of independent interest,  tells us that the maximal fluctuations for IDLA on $G\times \bbZ$ from the rectangular shape are $\mathcal{O}(\sqrt{\t_N} ( \log N )^2 )$ for polynomially many steps, when starting the process from flat. 
We split the proof into several lemmas, keeping 
Assumption \ref{assumption}  in force throughout. \\

Let us start with the following coupon-collector inequality stating that, starting from flat, after $\mathcal{O}(N\log N)$ releases level $1$ is filled with high probability. 
\begin{Lemma}\label{le:coupon}
Let $(A(t))_{t\geq 0} $ be an IDLA process on $\mathbb{G}$ with $A(0) = R_0$. For arbitrary $\g >0$, let $a_\g = (\gamma +1)/\d  $. Then we have that
	\[ \bbP \Big( A (a_\g N \log N ) \supseteq R_1 \Big) 
	\geq 1-N^{-\g } \]
for $N$ large enough. 
\end{Lemma}
\begin{proof}
By Assumption \ref{assumption}, each time a new random walk is released from level $0$ each vertex at level $1$ has probability at least 
	\[ \min_{v\in V } \pi_N(v) \geq \frac \d N \]
of being filled, if empty. Thus 
	\[ \bbP (A (a_\g N\log N ) \nsupseteq R_1) 
	\leq \sum_{v\in V} \bbP ((v,1) \notin A(a_\g N \log N ) ) 
	\leq N e^{- (\g +1 ) \log N } = N^{-\g } . 
	\]
\end{proof}
We use this simple observation to give a coupling proof of Theorem \ref{th:Tpoly}. 
Our proof  generalises the one of  Theorem 1.1 in \cite{levine2018long}, bypassing the lack of the logarithmic fluctuations result via Lemma \ref{le:coupon}. 
\begin{proof}[Proof of Theorem \ref{th:Tpoly}] 
The proof consists of two steps. In Step 1 we use Lemma \ref{le:coupon} together with the Abelian property of IDLA to reduce the analysis to the case $T \leq  \a N \sqrt{\tau_N} (\log N)^2$ for some finite constant $\a$. In Step 2 we then bound the fluctuations of $A(T)$ for such small $T$ by a standard argument due to Lawler et al \cite{lawler1992internal}. \\

\emph{Step 1.} 
If $T\leq \a N \sqrt{\tau_N} (\log N)^2 $ for some $\a < \infty $ then skip to Step 2. 
Otherwise, fix $\g >0$ as in the statement. Let $a_{\g +2m}$ be defined as in Lemma \ref{le:coupon}, and $b_{\g +m+1 , m}$ be as in Proposition \ref{pr:coupling}. We write $a,b$ in place of $a_{\g+ 2m}$, $b_{\g +m+1, m}$ for brevity.  
Define 
	\[ \ell := \left\lfloor \frac{T}{ a N \log N } \right\rfloor, \qquad 
	 t_0 := (a N \log N ) \ell ,\]
and note that $ b \sqrt{\tau_N} \log N +1 \leq \ell \leq N^{m-1} $. 
At cost of increasing $a$ and $b$, we may assume that $t_0$ and $b \sqrt{\tau_N} \log N$ are integers. 

Recall that the IDLA process starts from the flat configuration $A(0) = R_0$. To start with, release $aN \log N $ walkers from level $0$ according to $\pi_N$. We stop the walkers upon reaching level $1$. A walker that reaches level $1$ at an empty site settles there and never moves again, and the site is marked as occupied. Walkers that reach level $1$ at occupied sites, on the other hand, are stopped and declared \emph{frozen}. As the name suggests, the motion of frozen walkers will be resumed at a later stage. Then, by Lemma \ref{le:coupon}, after 
$a N \log N $ releases level $1$ is filled with probability at least $1-N^{-(\g +2m)}$ and all walkers, whether settled or frozen, are at level $1$. On this high probability event, we repeat this procedure by releasing $a N \log N$ additional walkers which stop, with the same rule as before, at level $2$, and so on. Let $W_1$ denote the cluster of occupied sites after releasing $t_0$ walkers from level $0$ in groups of $aN \log N$, stopping walkers of the $k^{th}$ group at level $k$ as described above. Then 
	\[ \bbP \big( W_1 = R_{\ell} \big) 
	\geq 1- \ell  N^{-(\gamma+2m)} 
	\geq 1-N^{-(\g +m+1)} \]
for $N$ large enough.  
We stress that, while $W_1$ is not distributed as an IDLA cluster, the Abelian property implies that unfreezing all the frozen walkers and letting them move until settling at empty sites will result in a cluster with the same distribution as $A(t_0)$. 

On the high probability event $\{ W_1 = R_\ell\}$ all the frozen walkers between levels 
$1$ and $\ell - b\sqrt{\tau_N} \log N$ are at distance at least $b\sqrt{\tau_N} \log N$ from the boundary of $W_1$, so by Proposition \ref{pr:coupling} releasing them has, with high probability, the same effect as releasing the same number of walkers from level $0$ according to $\pi_N$. 
To make this precise, note that on $\{ W_1 = R_\ell\}$ there are exactly 
	\[ t_1 := (aN \log N -N)( \ell - b\sqrt{\tau_N} \log N ) \]
frozen walkers up to level $\ell - b\sqrt{\tau_N} \log N$. Let $W_1(t_1)$ denote the cluster obtained by releasing these $t_1$ frozen walkers. 
Then, if $(A'(t))_{t\geq 0}$ is an auxiliary IDLA process with $A'(0) = R_\ell$, by Proposition \ref{pr:coupling} we can couple $A'(t_1)$ and $W_1 (t_1)$ so that 
	\[ \bbP ( W_1 = R_\ell , A'(t_1) = W_1(t_1) ) \geq 1- 2N^{-(\gamma+m+1)} \]
for $N$ large enough. This suggests a way to couple the original IDLA process $A$ with the auxiliary IDLA process $A'$. Indeed, having built $A'(t_1)$ as above, we can add $T_1 = T - N(\ell - b\sqrt{\tau_N} \log N )$ particles to $A'(t_1)$ using the locations of the frozen particles in $W_1(t_1)$ as starting locations. By the Abelian property, this gives 
	\[ \| A(T) - A'(T_1) \|_{TV} \leq \bbP (W_1 \neq R_\ell ) + \bbP (W_1 = R_\ell , A'(t_1) \neq W_1 (t_1) ) \leq 2N^{-(\g+m+1)}  \]
for $N$ large enough. It thus suffices to prove that \eqref{eq:max} holds with $A'(T_1)$ in place of $A(T)$. Moreover, since the process $A'$ started from $R_\ell$, the downshift of $A'(T_1)$ by $\ell$ levels has the same distribution as $A(T_1)$. This therefore reduces the problem to proving that \eqref{eq:max} holds with $T_1$ in place of $T$. Now, if $T_1 \leq \a N \sqrt{\tau_N} (\log N)^2$ for some $\a <\infty$ stop and move to Step 2. Otherwise, iterate this procedure. Note that each iteration reduces $T$ by at least $a N \log N - N $. Then, after $k< N^{m-1}$ iterations we are left with 
	\[ T_k \leq (a N \log N )(2b\sqrt{\tau_N} (\log N)^2 ), \] 
and 
	\[  \begin{split} 
		 \bbP \Big( & R_{\frac{T}{N} - C_{\g ,m} \sqrt{\t_N} ( \log N )^2 }
	\subseteq  A(T) \subseteq 
	R_{\frac{T}{N} + C_{\g,m} \sqrt{\t_N} ( \log N )^2 } \Big) 
	\geq 
	\\ & 	\geq  \bbP \Big(  R_{\frac{T_k}{N} - C_{\g ,m} \sqrt{\t_N} ( \log N )^2 }
	\subseteq  A(T_k) \subseteq
	R_{\frac{T_k}{N} + C_{\g,m} \sqrt{\t_N} ( \log N )^2 } \Big) -N^{-(\g +2)} , 
	\end{split}\]
for $N$ large enough.   At this point we can move to Step 2. \\

\emph{Step 2.} If $T\leq \a ' N (\log N)^2$ for some finite constant $\a '$ then the conclusion is trivial. Otherwise,  let $(A(t))_{t\geq 0}$ be an IDLA process starting from $A(0)=R_0$. We show that, if $T \leq \a N \sqrt{\t_N} (\log N )^2$ for some $\a$ (which may depend on $\g , m$), then there exists $\b \in (0,1)$, independent of $N$, such that 
	\begin{equation} \label{smallT}
	 \bbP \Big( A(T) \nsubseteq R_{\frac{3T}{N}} \Big) \leq \b^{T/N} 
	 \end{equation}
for $N$ large enough. This would conclude the proof, as the inner bound is trivial since $T/N \leq \a \sqrt{\tau_N} (\log N)^2$. 

To see \eqref{smallT}, let 
 $Z_k (t) := | A(t) \cap \{ y=k\} |$ denote the number of sites in $A(t)$ at level $k$, and set $\mu_k(t) := \bbE (Z_k(t))$. Then 
	\[ \bbP \Big( A(T) \nsubseteq R_{\frac{3T}{N}} \Big)
	 \leq \bbP \Big( Z_{\frac{3T}{N}+1}(T) \geq 1 \Big) 
	\leq \mu_{ \frac{3T}{N}+1} (T) . \]
We claim that 
	\begin{equation} \label{LBG_up}
	 \mu_k(t) \leq N \Big( \frac{1}{N} \Big)^{k-1} \frac{t^{k}}{k!} 
	 \end{equation}
for all $k>0$ and $t\geq 0$. 
Indeed, clearly $\mu_1(t) \leq N$ and $\mu_k(0) =0$  for all $k>0$. For other values of $k,j$  we have 
	\[ \begin{split} 
	\mu_k(t+1) - \mu_k(t) & = \bbE ( Z_k (t+1)  - Z_k (t) ) 
	 \leq \frac{\d'}{N} \bbE ( | A(t) \cap \{ y=k-1\} | ) 
	= \frac{\d'}{N} \mu_{k-1} (t) , 
	\end{split} \]
where in the above inequality we have used that the probability that the $(t+1)^{th}$ walker reaches level $k-1$ inside $A(t)$ is maximised when $A(t)$ is completely filled up to level $k-2$. Thus for $k>0$ 
	\[ \mu_k ( t) = \sum_{s=0}^{t-1} ( \mu_k(s+1) - \mu_k(s) ) \leq 
	\frac{\d'}{N} \sum_{s=0}^{t-1} \mu_{k-1}(s) , \]
and \eqref{LBG_up} follows by a simple iteration. 
Now take $t=T$, $k =\frac{3T}{N}+1$ and recall that $k! \geq k^k e^{-k}$ to get 
	\[ \mu_{\frac{3T}{N}+1} (T) 
	\leq N \Big( \frac{1}{N} \Big)^{\frac{3T}{N} } 
	\frac{T^{\frac{3T}{N}+1} }{(\frac{3T}{N} +1)!} 
	\leq N^2 \Big[ \Big( \frac{e}{3} \Big)^3 \Big]^{\frac{T}{N}}  \leq \b^{T/N} 
	\] 
for any $  \b \in \big(\big(\frac{e}{3}\big)^3 , 1)$ and $N$ large enough. This proves \eqref{smallT}, hence concluding the proof of Theorem \ref{th:Tpoly}. 
\end{proof}

\section{Proof of Theorems \ref{th:typical_height} and \ref{th:main_teo}}\label{sec:brief}
We sketch here the proofs of Theorems \ref{th:typical_height} and \ref{th:main_teo}.  We give precise statements of the several lemmas leading to the proof, but choose to leave the proof details to the reader, since these are straightforward generalisations of the arguments in \cite{levine2018long}. This section also serves as a survey of the ideas introduced for the case $G = \bbZ_N$ in \cite{levine2018long}. 
We split the argument into several steps, for ease of reading. \\

\emph{Step 1: Decay of excess height.} 
For an IDLA cluster $A \in \Omega$, let $|A|$ count the number of vertices in $A$ above level $0$. We define the \emph{excess height} of $A$ by 
	\[ \cE (A) = h(A) - \frac{|A|}{N} ,\]
where $h(A)$ denotes the height of $A$. Thus $\cE (R_k) =0$ for all infinite rectangles $R_k$, while $\cE(A)$ is large if $A$ is tall and sparse. To start with, we show that if the excess height of a cluster is too large, it tends to decrease under IDLA dynamics. 
\begin{Lemma}\label{le:lowE}
Let $(A(t))_{t\geq 0}$ be an IDLA process on $G\times \bbZ$ with $A(0) \supseteq R_0$ and filtration $(\cF_t )_{t\geq 0}$. For $t\geq 0$ let $\cE (t) = \cE (A(t))$ denote the excess height of $A(t)$. Then there exists a constant $C_\cE$, depending only on $\d$,  such that if $\cE(t) > C_\cE N  \sqrt{\tau_N} (\log N)^2$ then 
	\[ \bbE(\cE (t+1) - \cE (t) | \cF_t ) < - \frac 1 {2N} \]
for $N$ large enough. 
\end{Lemma}
To see this, note that if the excess height is large, then there are many empty sites below level $h(t)$. Call a level below $h(t)$  \emph{bad} if it contains at least one empty site. If the excess height is high enough, then there are many bad levels at distance at least $20 \sqrt{\tau_N} \log N$ one another. Thus a random walker has time to mix in the horizontal coordinate between bad levels, and each time it reaches a new bad level it has probability at least $\d /(2N)$ to settle at an empty site. It follows that, if there are enough bad levels, the random walk will settle below height $h(t)$ with high probability, which will cause the excess height to decrease. We refer the reader to \cite{levine2018long}, Lemma 6.2, for details. 

Lemma \ref{le:lowE} shows that an IDLA process spends a small proportion of time in clusters with high excess height. From this it is standard (cf.\ \cite{levine2018long}, Proposition 6.1) to deduce that stationary clusters have low excess height, which gives the following result. 

\begin{Lemma}\label{le:lowEstationary}
For any $\g >0$ there exists a constant $C_{\cE,\g}$, depending only on $\g , \d $, such that, with $\cE^*_\g = C_{\cE, \g  } \sqrt{\tau_N } N (\log N)^2 $, 
	\[ \mu_N ( \{ A\in \Omega : \cE (A) > \cE^*_\g  \} \big) 
	\leq N^{-\g } \]
for $N$ large enough. 
\end{Lemma}
\smallskip

\emph{Step 2: From low excess height to low height.}
Let $(A(t))_{t\geq 0}$ be an IDLA process starting from a stationary configuration $A(0) \sim \mu_N$, and denote the associated shifted IDLA process by $(A^*(t))_{t\geq 0}$. While Lemma \ref{le:lowEstationary} shows that $A(0)$ has low excess height with high probability, this does not imply low height, since $|A(0)|$ is random. It gives, on the other hand, an upper bound on the number of empty sites, that is the number of sites in $R_{h(A(0))} \setminus A(0)$. 
This allows us to show that the height of the associated shifted IDLA process becomes $\cO (N^4 \sqrt{\tau_N}  ) $ within polynomial time, as stated below. 
\begin{Lemma}\label{le:lowEtolowH}
Let $(A^*(t))_{t\geq 0}$ be a shifted IDLA process with $A(0) \sim \mu_N$, and let 
	\[ T^* = \inf\{ t\geq 0 : h(A^*(t)) \leq N^4 \sqrt{\tau_N} \} . \]
Then for any $\g >0$ there exists a constant $C_{\cE , \g }^* $, depending only on $\g , \d$, such that
	\[ \bbP ( T^* > N^6 \sqrt{\tau_N} ) \leq N^{-\g} \]
for $N$ large enough. 
\end{Lemma}
The above result is proved as follows. On the high probability event $\cE (A(0)) \leq \cE^*_{\g+2}$ there are at most $N \cE^*_{\g +2}$ empty sites below level $h(A(0))$. Each time we add a new particle to the cluster, the lowest empty site has probability at least $\d /N$ to be filled, independently of everything else. Hence it will take at most a geometric number of releases of parameter $\d /N$ to fill it. We fill empty sites sequentially while the excess height does not exceed $2\cE^*_{\g+2}$. When it does, we wait for the dynamics to bring the excess height below $\cE^*_{\g+2}$, and try again to fill all the empty sites. It is easy to see that, with high probability, this procedure succeeds within $N^6 \sqrt{\tau_N}$ particle releases. See \cite{levine2018long}, Proposition 6.2, for details. 
\begin{Remark}\label{rem:positiverec}
The same argument shows that shifted IDLA defines a positive recurrent Markov chain, since starting from the flat configuration $R_0$,  
the chain comes back to a certain finite set containing $R_0$ in finite time.
\end{Remark}

\emph{Step 3: Forgetting polynomially high initial configurations.} 
Lemma \ref{le:lowEtolowH} tells us that stationary clusters have $\cO (N^6 \sqrt{\tau_N})$ height with high probability. We use the polynomial height bound to show that an IDLA process starting from a stationary configuration forgets that it did not start from flat within polynomially many steps. 
\begin{Proposition}\label{pr:step3}
Let $A_0, A_0' \in \Omega $ be such that $|A_0| = |A'_0|$, $A_0' \supseteq R_{\lfloor |A_0|/N \rfloor }$  and 
	\[ h(A_0) \leq \alpha N^6 \sqrt{\tau_N} , \qquad 
	h(A_0') \leq \frac{|A_0|}{N} +1 \]
for some absolute constant $\alpha<\infty$. Let $(A(t))_{t\geq 0}$ and $(A'(t))_{t\geq 0}$ denote two IDLA processes starting from $A_0$ and $A_0'$ respectively, and denote the laws of $A(t)$, $A'(t)$ by $P(t) , P'(t)$. 
Then for any $\g >0$ there exists a constant $c_\g$, depending only on $\g , \d , \d' $, such that, writing $s_\g = \a N^7 \sqrt{\tau_N} + c_\g N \sqrt{\tau_N} (\log N )^2$ for brevity, we have 
	\[ \| P(s_\g ) - P'(s_\g ) \|_{TV} \leq N^{-\g} \]
for $N$ large enough. 
\end{Proposition}
To prove the above result we use the coupling ideas introduced in Section \ref{sec:big}. Indeed, let $W_0$ be a new IDLA cluster, independent of everything else, built by adding $s_\g $ particles to the flat configuration $R_0$. Let $\b_{\g +1, 9}$ be as in the statement of Proposition \ref{pr:coupling}, and write $\b$ in place of $ \b_{\g +1, 9}$ for brevity. 
By Theorem \ref{th:Tpoly} we can choose $c_\g$ large enough so that $W_0$ is completely filled up to height $\a N^6 \sqrt{\tau_N} + \b \sqrt{\tau_N} (\log N)^2$ with high probability. In particular, $W_0$ contains both $A(0)$ and $A'(0)$, and each vertex of $A(0) \cup A'(0)$ is at distance at least $b \sqrt{\t_N} (\log N)^2 $ from the boundary of $W_0$. 

Let $W(0)$ (respectively $W'(0)$) denote the cluster obtained by adding to $W_0$ a frozen particle at each site of $A(0)$ (respectively $A'(0)$). Then $W(0) , W'(0)$ contain the same number $T=|A_0|$ of frozen particles, that we release in pairs. To do this, fix an arbitrary enumeration of the locations of the frozen particles: 
	\[ z_1 , z_2 \ldots z_T \in W(0) , \qquad 
	 z'_1 , z'_2 \ldots z'_T \in W'(0) .\]
For each  $1\leq t\leq T$ we build $W(t) , W'(t)$ from $W(t-1) , W'(t-1)$ as follows. Start two simple random walks in $W(t-1) , W'(t-1)$ from $z_t , z'_t$. The walks are coupled in such a way that their trajectories meet within $W_0$ with probability at least $N^{-(\g+10)}$, as in the proof of Proposition \ref{pr:coupling}. They settle upon exiting $W(t-1) , W'(t-1)$ respectively, and their settling locations are then added to $W(t-1) , W'(t-1)$ to make $W(t) , W'(t)$. Note that, importantly, if $W(t-1) = W'(t-1)$ and the $t^{th}$ coupling is successful, then $W(t) = W'(t)$. 

This defines two auxiliary processes  $(W(t))_{t\leq T}$ and $(W'(t))_{t\leq T}$ on $\Omega$. We stress that these are not IDLA processes, due to the presence of the frozen walkers. Nevertheless, the Abelian property implies that 
		\[ A(s_\g ) \stackrel{(d)}{=} W(T) , \qquad 
	A'(s_\g ) \stackrel{(d)}{=} W'(T) ,\]
since after $T$ steps all the frozen walkers have settled. To finish the proof it then suffices to show that 
	\[ \| W(T) - W'(T) \|_{TV} \leq N^{-(\g +1)}\]
for $N$ large enough. But this follows from Proposition \ref{pr:coupling}, since $T=|A_0|$ is at most polynomial in $N$, and all the frozen walkers' starting locations are at distance at least $\b \sqrt{\t_N} (\log N)^2$ from the cluster boundary, thus leaving time for the walkers to mix before exiting. 
\medskip

\emph{Step 4: Stationary height: proof of Theorem \ref{th:typical_height}.} 
We use Proposition \ref{pr:step3} to argue that a stationary shifted IDLA process has, after polynomially many steps, both stationary law and $\mathcal{O}(\sqrt{\tau_N} (\log N)^2)$ fluctuations, since it forgot that it did not start from flat with high probability. This tells us that stationary clusters have height $\mathcal{O}(\sqrt{\tau_N}( \log N)^2)$, and it thus proves  Theorem \ref{th:typical_height}.

To expand, let $(A(t))_{t\geq 0}$ denote a shifted IDLA process started at stationarity $A(0) \sim \mu_N$. Then $A(t) \sim \mu_N$ for all $t\geq 0$. In particular, if for arbitrary $\g >0$ we define $s_\g $ as in Proposition \ref{pr:step3}, we have that 
	\[ A(s_\g ) \sim \mu_N. \]
On the other hand, since $h(A_0) =\mathcal{O}( N^6 \sqrt{\tau_N})$ with high probability by Lemma \ref{le:lowEstationary}, Proposition \ref{pr:step3} tells us that $A(s_\g)$ is close in distribution to a shifted IDLA process started from height at most $1$, and hence, using that  $s_\g \leq N^9$ for $N$ large enough, 
\[ h(A(s_\g ) ) \leq C_\g \sqrt{\tau_N} (\log N)^2 \]
with high probability by Theorem \ref{th:Tpoly}. 

\smallskip 

\emph{Step 5: Proof of Theorem \ref{th:main_teo}.}
For arbitrary $\g >0$, let $c_\g$ be defined as in Theorem \ref{th:typical_height}, and set 
	\[ \Omega_\g = \{ A \in \Omega : h(A) \leq c_\g \sqrt{\tau_N} (\log N)^2 \}. \]
Then $\mu_N(\Omega_\g ^c) \leq N^{-\g}$, and Theorem \ref{th:main_teo} follows by exactly the same argument used in the proof of Proposition \ref{pr:step3}.

\medskip

\bibliography{HLbib2}
\bibliographystyle{plain}

\end{document}